\documentclass[letterpaper]{article}
\usepackage{amsmath}
\usepackage{amssymb}
\usepackage{amsthm}
\usepackage{enumerate}
\usepackage{graphicx}
\usepackage{url}
\usepackage{color}
\usepackage[letterpaper]{geometry}

\makeatletter
\renewcommand\paragraph{\@startsection{paragraph}{4}{\z@}%
            {-2.5ex\@plus -1ex \@minus -.25ex}%
            {1.25ex \@plus .25ex}%
            {\normalfont\normalsize\bfseries}}
\makeatother
\numberwithin{equation}{section}

\newcommand{\Alpha}{A}

\newcommand{\Ker}{\textup{Ker}}
\newcommand{\ri}{\text{ri}}

\newcommand{\Zero}{{\mathbf{0}}}

\newcommand{\RR}{{\mathbb R}}

\newcommand{\cB}{{\mathcal B}}
\newcommand{\cC}{{\mathcal C}}

\newcommand{\cF}{{\mathcal F}}
\newcommand{\cG} {{\mathcal O}_{\mathcal{S}}}
\newcommand{\cOS}{{\mathcal O}_{\mathcal{S}}}

\newcommand{\cO}{{\mathcal O}}
\newcommand{\cP}{{\mathcal P}}
\newcommand{\cQ}{{\mathcal Q}}
\newcommand{\cS}{{\mathcal S}}

\newcommand{\ol}{\overline}

\renewcommand{\Im}{\textup{Im}}
\newcommand{\conv}{\textup{co}}

\newcommand{\cone}{\textup{cone}}

\newcommand{\rank}{\textup{rank}}

\newtheorem{theorem}{Theorem}

\newtheorem{remark}[theorem]{Remark}
\newtheorem{corollary}[theorem]{Corollary}
\newtheorem{problem}[theorem]{Problem}
\newtheorem{lemma}[theorem]{Lemma}

\begin{document}
\sloppy

\title{Some Results on an Affine Obstruction to Reach Control}
\author{Melkior Ornik\footnote{Department of Electrical and Computer Engineering, University of Toronto, Toronto ON Canada M5S 3G4, 
$\{\textrm{melkior.ornik@scg.utoronto.ca, broucke@control.utoronto.ca}\}$}, Mireille~E.~Broucke\footnotemark[1]}

\date{\today}
\maketitle

\begin{abstract}
This ArXiv paper is a supplement to \cite{OrnBro15a} and contains proofs of preliminary claims omitted in \cite{OrnBro15a} for lack of space. The paper deals with exploring a necessary condition for solvability of the Reach Control Problem (RCP) using affine feedback. The goal of the RCP is to drive the states of an affine control system to a given facet of a simplex without first exiting the simplex through other facets. In analogy to the problem of a topological obstruction to the RCP for continuous state feedback studied in \cite{OrnBro15a}, this paper formulates the problem of an affine obstruction and solves it in the case of two- and three-dimensional systems. An appealing geometric cone condition is identified as the new necessary condition.
\end{abstract}

\section{Introduction}
The Reach Control Problem (RCP), first introduced in \cite{HabVan01} and given a modern formulation in \cite{Habetal06, RosBro06}, is a fundamental problem in piecewise affine and hybrid system theory. A reach control approach has been shown to be useful in a number of applications, including aircraft and underwater vehicles \cite{Belta2}, genetic networks \cite{Belta3}, and aggressive maneuvers of mechanical systems \cite{SUPERMARIO}. Nevertheless, for a given system, it is still not known in general whether the RCP is solvable by either affine or continuous state feedback. This paper formulates the problem of an obstruction to solving the RCP by affine state feedback.

Consider an $n$-dimensional simplex $\mathcal{S} \subseteq \mathbb{R}^n$ with vertices $v_0,v_1,\ldots,v_n$. The facets of $\cS$ are denoted by $\mathcal{F}_0,\ldots,\mathcal{F}_n$, where each facet is indexed by the vertex it does not contain. The facet $\cF_0$ is called the {\em exit facet}. 
We consider an affine control system defined on $\cS$, given by
\begin{equation}
\label{syst}
\dot{x} = Ax + Bu + a \,.
\end{equation} 
The RCP asks the following question: Is it possible to find a state feedback $u : \cS \to \mathbb{R}^m$ such that, 
for any initial state $x_0 \in \cS$, the closed-loop trajectory leaves $\cS$ in finite time, 
and it does so by leaving through facet $\cF_0$? Let $\phi_u(t,x_0)$ be the trajectory of system \eqref{syst} under state feedback $u$ and with initial state $x_0$. 

\begin{problem}
\label{prob0}
Is it possible to find $u: \cS \to \mathbb{R}^m$ such that for each $x_0 \in \cS$, there 
exists $T>0$ such that
\begin{enumerate}[(i)]
\item 
$\phi_u(t,x_0) \in \cS$ for all $t\in\left[0,T\right]$,
\item 
$\phi_u(T,x_0) \in \mathcal{F}_0$,
\item 
$\phi_u(t,x_0) \not\in \cS$ for all $t \in \left(T,T+\varepsilon\right)$ for some $\epsilon>0$.
\end{enumerate}
\end{problem}

In this paper we focus on a necessary condition for solvability of the RCP using affine feedback. In particular, for an affine state feedback to solve the RCP, it must not admit any closed-loop equilibria
in $\cS$. Let $\cB=\Im(B)$. It is easily shown that the equilibria of \eqref{syst} can only lie in the affine subspace 
$$
\cO = \{ x \in \mathbb{R}^n ~|~ Ax+a \in \cB \} \,.
$$ 
As we are only interested in potential equilibria contained in $\cS$, we study the set 
$$
\cOS = \mathcal{S} \cap \cO \,.
$$

We are interested in seeing whether we can design an affine feedback on $\cOS$ satisfying the conditions of Problem \ref{prob0}. This is clearly a necessary condition for solvability of the RCP.

For each $x\in\mathcal{S}$, we define the inside pointing cone $\mathcal{C}(x)$ with respect to $\mathcal{S}$ by 
\begin{equation}
\label{cone}
\mathcal{C}(x)=\{y\in\mathbb{R}^n|h_j\cdot y\leq 0 \textrm{ for all } j\in\{1,\ldots,n\}\backslash I(x)\}\textrm{,}
\end{equation} where $h_j$ is an outward pointing normal to $\cF_j$, and $I(x)=\{i_1,i_2,\ldots,i_k\}$ is the minimal set of indices of vertices of $\mathcal{S}$ such that $x\in\mathrm{co}\{v_{i_1},\ldots,v_{i_k}\}$. In other words, $x$ is in the interior of $x\in\mathrm{co}\{v_{i_1},\ldots,v_{i_k}\}$. 

$\mathcal{C}(x)$ contains all vector directions that, when appended to $x$, point inside $\cS$ or through $\cF_0$. Hence, by the conditions of Problem \ref{prob0}, $f(x)=Ax+Bu(x)+a\in\mathcal{C}(x)$ for all $x\in\cS$ is a necessary condition for solvability of the RCP. If $u$ is an affine state feedback, $f$ is affine as well. We also note that for $x\in\cOS$, $Ax+a\in\mathcal{B}$, and hence, $f(x)\in\mathcal{B}$ for $x\in\cOS$.

Thus, given the above observations, we want to study the following problem:

\begin{problem}
\label{proa}
Let $\cOS$, $\mathcal{S}$ and $\mathcal{B}$ be as above. Does there exist an affine map $f:\cOS\to\mathcal{B}$ that satisfies
\begin{enumerate}[(i)]
\item $f(x)\in\mathcal{C}(x)$ for all $x\in\cOS$,
\item $f(x)\neq 0$ for all $x\in\cOS$?
\end{enumerate}
\end{problem}

The second condition implies that the system given by $\dot{x}=f(x)$ contains no equilibria in $\cS$. This problem has a continuous analogue studied in \cite{OrnBro15a, OrnBro15b}. The continuous analogue is referred to as the {\em topological obstruction problem in the RCP}, as stated below.

\begin{problem}
\label{proc}
Let $\cOS$, $\mathcal{S}$ and $\mathcal{B}$ be as above. Does there exist a continuous map $f:\cOS\to\mathcal{B}\backslash\{0\}$ such that, for every $x\in\cOS$, $f(x)\in\mathcal{C}(x)$?
\end{problem}

Section \ref{prel} of this paper contains preliminary results for several special cases. These apply equally to Problem \ref{proa} and to Problem \ref{proc}. Hence, this paper serves as a supplement to \cite{OrnBro15a}, providing simple proofs omitted in \cite{OrnBro15a} for lack of space. 
In Section \ref{aff}, we provide a solution to Problem \ref{proa} in the cases of $n=2$ and $n=3$ 
using a linear algebra approach.

\section{Preliminaries}
\label{prel}
In this section we introduce a sufficient condition for solvability of Problems \ref{proa} and \ref{proc}. Furthermore, we investigate the cases of $\dim\cG=0$ and $\dim\cG=n$. All of the following results apply both to Problem \ref{proa} and \ref{proc}. 

\begin{lemma}[Vertex Deletion]
\label{cla}
Let $I(p)=\{0,i_1,i_2,\ldots,i_k\}$, with $k\geq 0$. Furthermore, let $I(q)=\{i_1,i_2,\ldots,i_k,i_{k+1},\ldots,v_{i_l}\}$, where $l\geq k$. We take all $i_j$'s to be different, and all different from $0$. Then $\mathcal{C}(p)\subseteq\mathcal{C}(q)$.
\end{lemma}
\begin{proof}
By the definition of $\mathcal{C}$ in \eqref{cone}, 
$$\mathcal{C}(p)=\{y\in\mathbb{R}^n|h_j\cdot y\leq 0 \textrm{ for all } j\in\{1,\ldots,n\}\backslash\{0,i_1,i_2,\ldots,i_k\}\}\textrm{.}$$ 
On the other hand,
$$\mathcal{C}(q)=\{y\in\mathbb{R}^n|h_j\cdot y\leq 0 \textrm{ for all } j\in\{1,\ldots,n\}\backslash\{i_1,i_2,\ldots,i_k,i_{k+1},\ldots,i_l\}\}\textrm{.}$$
Thus, as it is clear that the set of constraints in $\mathcal{C}(q)$ is a subset of the set of constraints in $\mathcal{C}(p)$, $\mathcal{C}(p)\subseteq\mathcal{C}(q)$.
\end{proof}

\begin{remark}
From the proof, it is clear that it does not matter if $I(p)$ includes $0$ or not. Analogously, it does not matter if $q$ is in the convex hull of vertices that include $v_0$ or not.
\end{remark}

The above lemma can now be used to show that cones of points on the interior of a polytope in $\cG$ are less restrictive than cones of points at its boundary. This is given in Lemma \ref{ter}, and such a claim will be useful both for Problem \ref{proa} discussed in Section \ref{aff}, as well as in Problem \ref{proc} discussed in \cite{OrnBro15a}.

\begin{lemma}
\label{ter}
Let $\mathcal{H}\subseteq\cOS$ be a polytope, and let $x$ be any point in its interior: $x\in\mathrm{Int}(\mathcal{H})$. Also, let $y$ be any point on its boundary: $y\in\partial\mathcal{H}$. Then, $\mathcal{C}(y)\subseteq\mathcal{C}(x)$.
\end{lemma}
\begin{proof}
Consider the line going through points $x$ and $y$. As $x\notin\partial\mathcal{H}$, by extending that line past $x$, we can determine a point $z\in\partial\mathcal{H}$ such that $x=\alpha y+\beta z$, where $\alpha,\beta>0$, $\alpha+\beta=1$. Let us assume that $y=\sum_{i=0}^n \alpha_i v_i$, $z=\sum_{i=0}^n \beta_iv_i$. Since both $y$ and $z$ are in $\cOS\subseteq\mathcal{S}$, all $\alpha_i$'s and $\beta_i$'s are nonnegative. Then, $x=\sum_{i=0}^n (\alpha\alpha_i+\beta\beta_i)v_i$. Now, for any $i$, if $i\in I(y)$, then $\alpha_i\neq 0$. We notice that, in that case, no matter what $\beta$, $\beta_i$ and $\alpha$ are, $\alpha\alpha_i+\beta\beta_i>0$. Thus, $i\in I(x)$. In other words, $I(y)\subseteq I(x)$ and thus, by Lemma \ref{cla}, $\mathcal{C}(y)\subseteq\mathcal{C}(x)$.
\end{proof}

From now on, we will use the following notation: $$\mathrm{cone}(\cOS)=\bigcap_{x\in\cOS}\mathcal{C}(x)\textrm{.}$$ We note that by Lemma \ref{ter} $$\mathrm{cone}(\cOS)=\bigcap_{i=1}^r\mathcal{C}(o_i)\textrm{,}$$ where $o_1,\ldots,o_r$ are vertices of $\cOS$.

The following result provides a sufficient condition for solving Problems~\ref{proa} and \ref{proc}. Our discussion in Section \ref{aff} will show that this condition is not necessary in general. However, it holds a central position in treatment of a number of subcases when solving Problems \ref{proa} and \ref{proc} for $n=2,3$.

\begin{lemma}
\label{suf}
If $$\mathrm{cone}(\cOS)\cap\mathcal{B}\neq\{0\}\textrm{,}$$ then the answer to Problems \ref{proa} and \ref{proc} is affirmative.
\end{lemma}
\begin{proof}
Let $b\in\mathrm{cone}(\cOS)\cap\mathcal{B}\backslash\{0\}$. We note that, by definition of $\mathrm{cone}(\cOS)$, $b$ satisfies the inward-pointing condition at every point in $\cOS$. Thus, the function $f:\cOS\to\mathcal{B}\backslash\{0\}$ defined by $f(x)=b$ for all $x\in\cOS$ satisfies all the criteria of Problem \ref{proa} and of Problem \ref{proc}.
\end{proof}

As a dual of sorts to Lemma \ref{suf}, Lemma \ref{nec} gives an easy necessary condition for Problems \ref{proa} and \ref{proc}.

\begin{lemma}
\label{nec}
Assume that the function $f$ from Problem \ref{proa} (Problem \ref{proc}) exists. Then, for every $x\in\cOS$, there exists $0\neq b\in\mathcal{C}(x)\cap\mathcal{B}$.
\end{lemma}
\begin{proof}
For any such $x$, take $b=f(x)$. By the conditions of Problems \ref{proa} and \ref{proc}, $f(x)\in\mathcal{B}\backslash\{0\}$ and $f(x)\in\mathcal{C}(x)$.
\end{proof}

Finally, let us note that $\cOS$ is a manifold (in fact, a polytope) of dimension $0\leq \dim\cOS \leq n$. Cases $\dim\cOS=0$ and $\dim\cOS=n$, as well as the case of $v_0\in\cOS$, prove to be particularly easy to analyze. We do that as follows:

\begin{lemma}
\label{nul}
If $\dim{\cOS}=0$, the answer to Problems \ref{proa} and \ref{proc} is affirmative if and only if $$\mathrm{cone}(\cOS)\cap\mathcal{B}\neq\Zero\textrm{.}$$
\end{lemma}
\begin{proof}
We note that in this case, $\cOS$ consists of a single point $x\in\mathcal{S}$. Thus, $\mathrm{cone}(\cOS)=\mathcal{C}(x)$, sufficiency is proved by Lemma \ref{suf}, and necessity is proved by Lemma \ref{nec}.
\end{proof}

\begin{lemma}
\label{tec}
If $v_0\in\cOS$, then $\mathrm{cone}(\cOS)=\mathcal{C}(v_0)$ and the answer to Problems \ref{proa} and \ref{proc} is affirmative if and only if $$\mathrm{cone}(\cOS)\cap\mathcal{B}=\mathcal{C}(v_0)\cap\mathcal{B}\neq\Zero\textrm{.}$$
\end{lemma}
\begin{proof}
Sufficiency is proved by Lemma \ref{suf}. Now, by the Vertex Deletion Lemma, $\mathcal{C}(v_0)\subseteq\mathcal{C}(x)$ for all $x\in\cOS$. Thus, $$\mathcal{C}(v_0)\supseteq\mathrm{cone}(\cOS)=\bigcap_{x\in\cOS}\mathcal{C}(x)\supseteq\mathrm{cone}(v_0)\textrm{.}$$ So, $\mathrm{cone}(\cOS)=\mathcal{C}(v_0)$, and necessity thus follows from Lemma \ref{nec}.
\end{proof}

\begin{corollary}
\label{ful}
If $\dim{\cOS}=n$, then $\mathrm{cone}(\cOS)=\mathcal{C}(v_0)$ and the answer to Problems \ref{proa} and \ref{proc} is affirmative if and only if $\mathrm{cone}(\cOS)\cap\mathcal{B}=\mathcal{C}(v_0)\cap\mathcal{B}\neq\Zero$.
\end{corollary}
\begin{proof}
We note that $\dim{\cOS}=n$ implies $\cOS=\mathcal{S}\ni v_0$. The claim follows from Lemma \ref{tec}.
\end{proof}

In the remainder of the text, as well as in \cite{OrnBro15a}, we assume that $1\leq\dim\cOS\leq n-1$.

%

\section{Affine Case}
\label{aff}

This section contains the main contribution of this paper: we will solve Problem \ref{proa} in the case of $n=2,3$. We will do that on a case by case basis, employing methods from linear algebra. We note that the results from the previous section solved the cases of $\dim \cOS\in\{0,n\}$. This reduces the problem to $\dim\cOS=1$ and $\dim\cOS=2$ (when $n=3$). In both of these, the sufficient condition $$\cone(\cOS)\cap\cB\neq\Zero$$ from Lemma \ref{suf} will again make an appearance, as it will be shown that, depending on the case, Problem \ref{proa} is either always solvable, or the condition from Lemma \ref{suf} is a necessary condition.

\subsection{$n=2$}
We note that the case where $\dim\cOS\in\{0,2\}$ has been solved in Lemma \ref{nul} and Corollary \ref{ful}. Thus, the only remaining case is when $\dim\cOS=\dim\mathcal{B}=1$. However, this was covered in Theorem 1 of \cite{SemBro14}: the same Intermediate Value Theorem argument holds for both continuous and affine functions. Thus, $f$ from Problem \ref{proa} exists if and only if $\mathrm{cone}(\cOS)\cap\mathcal{B}\neq\Zero$.

\subsection{$n=3$}
Again, the cases in which $\dim\cOS\in\{0,3\}$ have been solved in Lemma \ref{nul} and Corollary \ref{ful}. Thus, the remaining cases are $\dim\cOS,\dim\cB\in\{1,2\}$. 

$\dim\cOS=\dim\mathcal{B}=1$ is, as above, covered in \cite{SemBro14}. Problem \ref{proa} is again solvable if and only if $\mathrm{cone}(\cOS)\cap\mathcal{B}\neq\Zero$.

\subsubsection{$\dim\cOS=1$, $\dim\mathcal{B}=2$}
The following two lemmas were stated and proved in \cite{OrnBro15a}. For the benefit of the reader, we repeat the proofs. Lemma \ref{l22} solves Problem \ref{proa} in the case of $\dim\cOS=1$ and $\dim\mathcal{B}=2$.

\begin{lemma}
\label{l12}
Suppose $\cG = \conv \{ o_1,\ldots,o_{\kappa+1} \}$ where the $o_i$'s are the vertices of $\cG$.
If there exists a linearly independent set $\{ b_i \in \mathcal{B} \cap \mathcal{C}(o_i) ~|~ i = 1,\ldots,\kappa+1 \}$, 
then the answer to Problems~\ref{proa} and \ref{proc} is affirmative. 
\end{lemma}

\begin{proof}
Let $f:\cG \to \mathcal{B}$ be defined by 
$f( \sum_{i = 1}^{\kappa+1} \alpha_i o_i) = \alpha_i b_i$, 
where $\sum \alpha_i = 1$ and $\alpha_i \ge 0$. Necessarily $f(x) \neq 0$ for $x \in \cG$ for otherwise 
the $b_i$'s would be linearly dependent. Also, by a standard convexity argument $f(x) \in \cC(x)$, $x \in \cG$.
\end{proof}

\begin{lemma}
\label{l22}
Let $n=3$, $\dim\mathcal{B}=2$, and let $o_1$ and $o_2$ be vertices of $\cG$. Then there exist 
linearly independent vectors $\{ b_1, b_2 ~|~ b_i \in \cB \cap \mathcal{C}(o_i) \}$. 
Moreover, if $\cG = \conv \{ o_1, o_2 \}$, the answer to Problems~\ref{proa} and \ref{proc} is affirmative.
\end{lemma}

\begin{proof}
First we assume $o_1 \in \ri(\cF_i)$ for some $i \in \{ 0,1,2,3\}$. By the definition of $\cC(o_1)$, it is a 
closed half space or $\RR^3$, so there exist linearly independent vectors $b_{11}, b_{12} \in \cB \cap \cC(o_1)$. 
We claim $\cB \cap \cC(o_2) \neq \Zero$. If $o_2 \in \ri(\cF_i)$ for some $i \in \{ 0,1,2,3\}$ then the argument above
proves the claim. Instead, assume w.l.o.g. that $o_2 \in \mathcal{F}_1 \cap \mathcal{F}_2$. Then 
$\mathcal{C}(o_2) = \{y\in\mathbb{R}^3|h_1\cdot y\leq 0, h_2\cdot y\leq 0\}$. Let $\cB = \Ker(M^T)$ for some 
$M \in \RR^{m \times n}$. Finding $0 \neq y \cB \cap \cC(o_2)$ is equivalent to solving 
\begin{equation}
\label{eq1}
\left[
\begin{array}{c}
h_1^T \\
h_2^T \\
\Alpha^T \end{array} \right]y=\left[\begin{array}{c}
s_1 \\
s_2 \\
0 \end{array} 
\right]\end{equation}
where $s_1,s_2 \in \mathbb{R}_0^-$ are unknown and $y\neq 0$. Because $\{ h_1, h_2 \}$ are linearly independent,
$\rank(H) \ge 2$. If $\rank(H) = 3$, then let
\[
[ y_1 ~~y_2 ] = H^{-1}
\left[\begin{array}{cc}
-1 &  0 \\
 0 & -1 \\
 0 &  0 
\end{array} 
\right] \,.
\]
Since $(-1,0,0)$ and $(0,-1,0)$ are linearly independent, $y_1$ and $y_2$ are linearly independent as well.

Next, assume $\rank(H) = 2$. In other words, $\Alpha=c_1h_1+c_2h_2$ for some $c_1,c_2\in\mathbb{R}$. Then, by taking $s_1=s_2=0$, equation \eqref{eq1} reduces to $$\left[\begin{array}{c}
h_1^T \\
h_2^T \\
\end{array} \right]y=0\textrm{.}$$
Now, by the rank-nullity theorem, the dimension of the kernel of the $2\times 3$ matrix on the left is $1$. Thus, there exists a nontrivial $y$ satisfying this equation. We make note of the fact that, if we take $v_0$ to be the origin, such a $y$ satisfies $y\in\mathcal{F}_1\cap\mathcal{F}_2=\mathrm{co}\{v_0,v_3\}$. In fact, since we can take any $y$ in this intersection of planes, we can take $y=v_3$.

We are now almost done. We assumed that $o_1$ is in the interior of one of the facets (or on the edges of $\mathcal{F}_0$, excluding the vertices). We proved that then there exist linearly independent $b_{11},b_{12}\in\mathcal{C}(o_1)$. We also proved that there exists a nontrivial $b_2\in\mathcal{C}(o_2)$. We claim that at least one of the pairs $\{b_{11},b_2\}$ and $\{b_{12},b_2\}$ will be linearly independent. Otherwise, $b_{11}$ is a scalar multiple of $b_2$, and $b_2$ is a scalar multiple of $b_{12}$. This is a contradiction with $b_{11}$ and $b_{12}$ being linearly independent. Thus, we have indeed found a linearly independent pair $b_1\in\mathcal{C}(o_1)$ and $b_2\in\mathcal{C}(o_2)$.

Let us now assume that neither $o_1$ nor $o_2$ are on the facet interiors (nor on the edges of $\mathcal{F}_0$, excluding the vertices). So, without loss of generality, $o_1\in\mathcal{F}_1\cap\mathcal{F}_2$ and $o_2\in\mathcal{F}_1\cap\mathcal{F}_3$. By the computations from several paragraphs above, we have shown that either there exist two linearly independent vectors in $\mathcal{C}(o_1)$ or $v_3\in\mathcal{C}(o_1)\textrm{.}$ Analogously, there either exist two linearly 
independent vectors in $\mathcal{C}(o_2)$ or $v_2\in\mathcal{C}(o_2)\textrm{.}$ Now, in the case there either exist two linearly independent vectors in $\mathcal{C}(o_1)$ or in $\mathcal{C}(o_2)$, the procedure in the previous paragraph generates a linearly independent pair of vectors, one in each of the cones.

In the remaining case, $v_3\in\mathcal{C}(o_1)$ and $v_2\in\mathcal{C}(o_2)$. Since $v_3$ and $v_2$ are obviously linearly independent, we again found our required pair of linearly independent vectors.
Finally, if $\cG = \conv \{ o_1, o_2 \}$, then by Lemma \ref{l12} the answer to Problem~\ref{proa} and Problem~\ref{proc} is affirmative.
\end{proof}

This answers Problem \ref{proa} whenever $\dim\cOS=1$.

If $\dim\cOS=2$ and $\dim\mathcal{B}=1$, the matter is clear: by the same argument in \cite{SemBro14}, which invokes the Intermediate Value Theorem, the vectors assigned at the segment between any two points $B$ need to be positive multiples of each other. Thus, all the cones $\mathcal{C}(x)$ for $x\in\cOS$ need to be the same. Hence, by Lemma \ref{suf} and Lemma \ref{nec}, the answer to Problem \ref{proa} is affirmative if and only if $\mathrm{cone}(\cOS)\cap\mathcal{B}\neq\Zero$.

\subsubsection{$\dim\cOS=2$ and $\dim\mathcal{B}=2$}
We assume that $v_0\not\in\cOS$, for that case has been settled by Lemma \ref{tec}. We also assume that $\mathrm{cone}(\cOS)\cap\mathcal{B}=\Zero$. Otherwise, we are done by Lemma \ref{suf}. Now, let us observe what $\cOS$ can look like. As given by the formula for product of simplices in \cite{Ket08}, $\cOS$ can either be a product of a $2$-simplex and a $0$-simplex, i.e., a triangle, or a product of two $1$-simplices, i.e., a quadrilateral. We also must allow for $\cOS$ passing through one of the vertices of $\mathcal{S}$, resulting in a triangle (essentially, a degenerated quadrilateral).

\paragraph{$\cOS$ is a triangle}
First, let us assume that $\cOS$ satisfies $o_i\in\left(v_0,v_i\right]$ for all $i=1,2,3$. Then, Theorem \ref{thm:symmetric} provides a solution to both Problem~\ref{proc} and Problem \ref{proa}. This theorem was also stated and proved in \cite{OrnBro15a}, but we provide both the statement and the proof for the benefit of the reader.

We first make note of a variant of Sperner's lemma from \cite{TALMAN}. The same variant was previously used in \cite{Bro10}.

\begin{lemma}
\label{lem:SpernerII}
Let $\cP = \conv \{ w_1,\ldots,w_{n+1} \}$ be an $n$-dimensional simplex.
Let $\{ \cQ_1, \ldots, \cQ_{n+1} \}$ be a collection of sets covering $\cP$ such that
\begin{itemize}
\item[(P1)]
Vertex $w_i \in \cQ_i$ and $w_i \not\in \cQ_j$ for $j \neq i$.
\item[(P2)]
If w.l.o.g. $x \in \conv \{ w_1, \ldots, w_l \}$ for some $1 \le l \le n+1$, then 
$x \in \cQ_1 \cup \cdots \cup \cQ_l$.
\end{itemize}
Then $\bigcap_{1 = 1}^{n+1} \ol{\cQ}_i \neq \emptyset$.
\end{lemma}

\begin{theorem}
\label{thm:symmetric}
Let $n=3$ and suppose $\cG = \conv \{ o_1, o_2, o_3 \}$ with $v_0 \not\in \cG$
and $o_i \in (v_0,v_i]$, $i = 1, 2, 3$.
The answer to Problem~\ref{proc} and Problem~\ref{proa} is affirmative if and only if 
$\cB \cap \cone(\cG) \neq \Zero$.
\end{theorem}

\begin{proof}
Sufficiency is provided by Lemma~\ref{suf}. For necessity, 
suppose there exists $f : \cG \rightarrow \cB \setminus \{ 0 \}$ such that $f(x) \in \cC(x)$, $x \in \cG$. 
By way of contradiction suppose $\cB \cap \cone(\cG) = \Zero$. 
Since $o_i \in (v_0,v_i]$, $i = 1, 2, 3$, we have 
\[
\cone(\cG) = \{ y \in \RR^n ~|~ h_j \cdot y \le 0 \,, j = 1, 2, 3 \} \,.
\]
Define the sets
\begin{equation}
\cQ_i := \{ x \in \cG ~|~ h_i \cdot f(x) > 0 \} \,, \quad i = 1, 2, 3 \,. 
\end{equation}
Now we verify the conditions of Lemma~\ref{lem:SpernerII}. 

Firstly, we claim that $\{ \cQ_i \}$ cover $\cG$. For suppose not. Then there exists $x \in \cG$ such that 
$h_j \cdot f(x) \le 0$, $j = 1, 2, 3$. Hence $f(x) \in \cB \cap \cone(\cG)$, so $f(x) = 0$, a contradiction 
to $f$ being non-vanishing on $\cG$.

Secondly, we verify property (P1). We claim that $o_i \in \cQ_i$ for $i = 1, 2, 3$. 
For suppose not. Then $h_i \cdot f(x) \le 0$. Additionally, because $f(o_i) \in \cC(o_i)$, 
$h_j \cdot f(x) \le 0$, $j \in \{ 1, 2, 3 \} \setminus \{ i \}$. We conclude 
$f(o_i) \in \cB \cap \cone(\cG)$, so $f(o_i) = 0$, a contradiction. 
Next we claim $o_i \not\in \cQ_j$, $j \neq i$. This is immediate since $f(o_i) \in \cC(o_i)$ implies
$h_j \cdot f(o_i) \le 0$, $j \neq i$. 

Thirdly, we verify property (P2). Suppose w.l.o.g. (by reordering the indices $\{ 1,2,3 \}$)
$x \in \conv \{ o_1,\ldots,o_r \}$ for some $1 \le r \le 3$. We claim $x \in \cQ_1 \cup \cdots \cup \cQ_r$.
For suppose not. Then $h_j \cdot f(x) \le 0$, $j = 1,\ldots,r$. Also, it is easily verified that
$\cC(x) = \{ y \in \RR^n ~|~ h_j \cdot y \le 0 \,, j = r+1,\ldots, 3 \}$. Thus, $h_j \cdot f(x) \le 0$, 
$j = r+1,\ldots,3$. 
Hence, $f(x) \in \cB \cap \cone(\cG)$, so $f(x) = 0$, a contradiction to $f$ being non-vanishing on $\cG$.

We have verified (P1)-(P2) of Lemma~\ref{lem:SpernerII}. Applying the lemma, there exists
$\ol{x} \in \bigcap_{i=1}^3 \ol{\cQ}_i$; that is, $h_j \cdot f(\ol{x}) \geq 0$, $j = 1,2,3$.
We conclude that $-f(\ol{x}) \in \cB \cap \cone(\cG)$, so $f(\ol{x}) = 0$, a contradiction. 
\end{proof}

From now on, we can assume that $\cOS$ is not a triangle satisfying the conditions of Theorem \ref{thm:symmetric}. Thus, if $\cOS$ is a triangle, by the discussion of simplicial products in \cite{Ket08}, it can either pass through one of the vertices of $\mathcal{S}$, or have all its vertices on the edges of $\mathcal{S}$ which connect a single vertex, say $v_1$, to the others.

In the latter case, say those vertices are $o_1\in\mathrm{co}\{v_1,v_2\}$, $o_2\in\mathrm{co}\{v_0,v_1\}$, $o_3\in\mathrm{co}\{v_1,v_3\}$. We assumed above that $\cOS$ does not pass through any of vertices $v_0$. Thus, by Lemma \ref{cla}, we note that the cones of all these three vertices (and hence of any point in $\cOS$, as the convex combination of $o_1$, $o_2$ and $o_3$ will have $v_1$ in its expansion in terms of vertices of $\mathcal{S}$) are subsets of $\mathcal{C}(o_1)$. Thus, $\mathrm{cone}(\cOS)=\mathcal{C}(o_1)$ and hence by Lemma \ref{suf} and Lemma \ref{nec}, the answer to Problem \ref{proa} is affirmative if and only if $$\mathcal{C}(o_1)\cap\mathcal{B}=\mathrm{cone}(\cOS)\cap\mathcal{B}\neq\Zero\textrm{.}$$

In the case where $\cOS$ passes through one of the vertices of $\mathcal{S}$, say without loss of generality that $o_1=v_1$, $o_2\in\mathrm{co}\{v_0,v_2\}$ and $o_3\in\mathrm{co}\{v_2,v_3\}$ (where neither $o_2$ nor $o_3$ coincide with any $v_i$'s). Now, as $\dim\mathcal{B}=2$, we know by Lemma \ref{l22} that there exist linearly independent vectors $b_1\in\mathcal{B}\cap\mathcal{C}(o_1)$ and $b_2\in\mathcal{B}\cap\mathcal{C}(o_2)$. Now, define $f:\cOS\to\mathcal{B}$ by $f(x)=Ax$, where $Ao_1=b_1$, $Ao_2=b_2$ and $Ao_3=b_2$. We first note that $o_1$, $o_2$ and $o_3$ are linearly independent. Thus, the above assignment can be accomplished. Next, we note that $\mathcal{C}(o_2)\subseteq\mathcal{C}(o_3)$ by Lemma \ref{cla}. Thus, the assignments on the vertices of $\cOS$ satisfy the cone condition.

Now, let us write any $x\in\cOS$ as $x=\alpha_1o_1+\alpha_2o_2+\alpha_3o_3$, where $\alpha_1+\alpha_2+\alpha_3=1$ and $\alpha_i$'s are nonnegative. We note that if $\alpha_i\neq 0$, $I(x)\supseteq I(o_i)$. Thus, by Lemma $\ref{cla}$, $\mathcal{C}(x)\supseteq \mathcal{C}(o_i)$. As $\mathcal{C}(x)$ is convex, then $\mathcal{C}(x)\supseteq \mathrm{co}\{\mathcal{C}(o_i):\alpha_i\neq 0\}$.  On the other hand, $Ax=\alpha_1Ao_1+\alpha_2Ao_2+\alpha_3Ao_3$. Thus, as we have proved that for every $i$, $Ao_i\in\mathcal{C}(o_i)$ $Ax$ is a convex combination of vectors from $\mathcal{C}(o_i)$, where $\alpha_i\neq 0$. We noted above that this means $Ax\in \mathcal{C}(x)$.

Note that the above proof works for any affine function: if it satisfies the cone criteria on the vertices, it will satisfy those criteria on the rest of $\cOS$ as well.

Finally, we note the following: with the above notation, $$f(x)=Ax=\alpha_1Ao_1+\alpha_2Ao_2+\alpha_3Ao_3=\alpha_1b_1+(\alpha_2+\alpha_3)b_2\textrm{.}$$ As $b_1$ and $b_2$ are linearly independent and $\alpha_1$ and $\alpha_2+\alpha_3$ can not both be zero, $f(x)\neq 0$ for any $x\in\cOS$. We have thus given a constructive solution for Problem \ref{proa} in this case.

\paragraph{$\cOS$ is a quadrilateral}

Say without loss of generality that $o_1\in\mathrm{co}\{v_0,v_2\}$, $o_2\in\mathrm{co}\{v_0,v_3\}$, $o_3\in\mathrm{co}\{v_1,v_2\}$ and $o_4\in\mathrm{co}\{v_1,v_3\}$ (where none of $o_i$'s actually coincide with any $v_j$'s). Now, we again know by Lemma \ref{l22} that there exist linearly independent vectors $b_1\in\mathcal{B}\cap\mathcal{C}(o_1)$ and $b_2\in\mathcal{B}\cap\mathcal{C}(o_2)$.

Now, from the definition of a cone in \eqref{cone}, we know that $b_1\cdot h_3\leq 0$ and $b_2\cdot h_2\leq 0$. We distinguish between two cases: in the first one, without loss of generality, $b_2\cdot h_2<0$.

Since $\{o_1,o_2,o_3\}$ is a linearly independent set, we know that there exist unique coefficients $\alpha_i$ such that $$o_4=\sum_{i=1}^3 \alpha_io_i\textrm{.}$$ Furthermore, $\alpha_2>0$, as $o_4=\lambda_1v_1+\lambda_3v_3$, with $\lambda_3>0$, and $v_3\not\in\mathrm{co}\{o_1,o_3\}$. Now, let us define $f:\cOS\to\mathcal{B}$ by $f(x)=Ax$, where $Ao_1=\varepsilon b_1$, $Ao_2=b_2$, $Ao_3=\varepsilon b_1$ (we note that $\mathcal{C}(o_1)\subseteq\mathcal{C}(o_3)$ by Lemma \ref{cla}), and $$\varepsilon=\frac{-\alpha_2(b_2\cdot h_2)}{|2(\alpha_1+\alpha_3)(b_1\cdot h_2)|}\textrm{.}$$ (If $(\alpha_1+\alpha_3)(b_1\cdot h_2)=0$, let $\varepsilon=1$.) As $\{o_1,o_2,o_3\}$ is a linearly independent set in $\mathbb{R}^3$, $A$ is well-defined.

Now, we note that, since $o_4=\alpha_1o_1+\alpha_1o_1+\alpha_1o_1$, $Ao_4=\varepsilon(\alpha_1+\alpha_3)b_1+\alpha_2b_2$ and thus $$f(o_4)\cdot h_2=\pm\frac{\alpha_2}{2}(b_2\cdot h_2)+\alpha_2(b_2\cdot h_2)\leq\frac{\alpha_2}{2}(b_2\cdot h_2)<0\textrm{.}$$ (If $(\alpha_1+\alpha_3)(b_1\cdot h_2)=0$, then $f(o_4)\cdot h_2=\alpha_2(b_2\cdot h_2)<0$.) Thus, the vector assigned to $o_4$ is in $\mathcal{C}(o_4)$. Thus, $f$ satisfies the cone condition at all four vertices of $\cOS$. By the proof same as in the case of the triangle, since $f$ is affine, it hence satisfies the cone condition at any point of $\cOS$.

Finally, we note that for any $x\in\cOS$, $f(x)=Ax=A(\kappa_1o_1+\kappa_2o_2+\kappa_3o_3+\kappa_4o_4)$, where $\kappa_1+\kappa_2+\kappa_3+\kappa_4=1$, and all $\kappa_i$'s are nonnegative. Thus, $f(x)=(\kappa_1+\kappa_3+\kappa_4\alpha_1+\kappa_4\alpha_3)\varepsilon b_1+(\kappa_2+\kappa_4\alpha_2)b_2$. Since $b_1$ and $b_2$ are linearly independent, $f(x)=0$ is thus equivalent to $\kappa_1+\kappa_3+\kappa_4\alpha_1+\kappa_4\alpha_3=0$ and $\kappa_2+\kappa_4\alpha_2=0$. Since $\alpha_2>0$, and $\kappa_i$'s are nonnegative, the latter equation implies $\kappa_2=\kappa_4=0$. Hence, the first equation implies $\kappa_1+\kappa_3=0$, which implies $\kappa_1=\kappa_3=0$. As $\kappa_1+\kappa_2+\kappa_3+\kappa_4=1$, this is clearly impossible. Thus, $f$ is nowhere zero. We are done, having defined a function satisfying Problem \ref{proa} on $\cOS$.

Now, let us assume that $b_1\cdot h_3=b_2\cdot h_2=0$. If we remind ourselves that $o_1\in\mathrm{co}\{v_0,v_2\}$ and $o_2\in\mathrm{co}\{v_0,v_3\}$, we also know (from the definition of cones at $o_1$ and $o_2$) that $b_1\cdot h_1,b_2\cdot h_1\leq 0$. Finally, we can assume that $b_1\cdot h_2,b_2\cdot h_3>0$. Otherwise, we would have that $b_1$ or $b_2$ is in $\mathrm{cone}(\cOS)\cap\mathcal{B}$, which was solved by Lemma \ref{suf}.

Now, let us first assume that $b_1\cdot h_1<0$ or $b_2\cdot h_1<0$. Without loss of generality we choose the first option. In that case, let $b_1'=b_1-cb_2$, where $$c=\frac{b_1\cdot h_1}{2(b_2\cdot h_1)}>0\textrm{.}$$ (If $b_2\cdot h_1=0$, take $c=1$ instead.) Now, we note that $b_1'\cdot h_1=b_1\cdot h_1-cb_2\cdot h_1=\frac{1}{2}b_1\cdot h_1<0$. (If $b_2\cdot h_1=0$, $b_1'\cdot h_1=b_1\cdot h_1<0$.) Also, $b_1'\cdot h_3=b_1\cdot h_3-cb_2\cdot h_3=-cb_2\cdot h_3<0$. Thus, $b_1'$ is in $\mathcal{C}(o_1)$ and in $\mathcal{B}$ (as it is a linear combination of vectors in $\mathcal{B}$). Furthermore, $b_1'$ and $b_2$ are still linearly independent and we already established $b_1'\cdot h_3<0$. Now, since we have that $b_1'\cdot h_3$ is strictly negative, we can go a few paragraphs back, just using $b_1'$ and $b_2$ instead of $b_1$ and $b_2$ in order to find a constructive answer to Problem \ref{proa}.

Finally, we have $b_1\cdot h_1=0$ and $b_2\cdot h_1=0$. However, we also know from before that $b_1\cdot h_3=b_2\cdot h_2=0$, and that $b_1\cdot h_2,b_2\cdot h_3>0$. Now, let $b_1'=b_1-b_2$. Then, $b_1'\cdot h_1=0$ and $b_1'\cdot h_3=b_1\cdot h_3-b_2\cdot h_3=-b_2\cdot h_3<0$. Thus, again, $b_1'$ is in $\mathcal{C}(o_1)$ and in $\mathcal{B}$, $b_1'$ and $b_2$ are linearly independent and $b_1'\cdot h_3<0$. Again, as before, we can go a few paragraphs back and obtain a solution to Problem \ref{proa}.

We note that this long and drawn out affair proved the following:

\begin{theorem}
\label{main2}
Let $n=3$, $\dim{\cOS}=2$, $\dim{\mathcal{B}}=2$ and $v_0\not\in\cOS$. Assume that $\cOS$ does not satisfy the conditions of Theorem \ref{thm:symmetric}. Then, the answer to Problem \ref{proa} is affirmative.
\end{theorem}

Thus, we've answered the last remaining case.

\end{document}